\documentclass[12pt]{article}
\usepackage[utf8]{inputenc}
\usepackage[T1]{fontenc}
\usepackage{mathpazo}

\usepackage{geometry}
\usepackage{amsmath}
\usepackage{amssymb}
\usepackage{amsthm}
\usepackage{mathtools}
\usepackage{braket}
\usepackage{commath}
\usepackage{tikz-cd}
\usepackage{booktabs}
\usepackage{afterpage}

\usepackage{hyperref}
\hypersetup{colorlinks = true}

\usepackage{comment}

\usepackage[style=alphabetic,useprefix,hyperref,backend=bibtex,maxcitenames=99,maxbibnames=99,maxalphanames=99]{biblatex}
\usepackage[T1]{fontenc}
\usepackage[utf8]{inputenc}
\usepackage[british]{babel}
\usepackage{amsthm}
\usepackage{amsfonts}
\usepackage{booktabs}
\usepackage{graphicx}
\usepackage{braket}
\usepackage{amsmath}
\usepackage{amssymb}
\usepackage{amscd}
\usepackage{mathtools}
\usepackage{latexsym}
\usepackage{tikz-cd}
\usepackage{titlesec}
\usepackage{hyperref}
\usepackage{xcolor}
\usepackage{verbatim}
\usepackage{stmaryrd}
\usepackage{changepage}
\usepackage{MnSymbol}

\theoremstyle{plain}
\newtheorem{proposition}{Proposition}[section]
\newtheorem{lemma}[proposition]{Lemma}
\newtheorem{corollary}[proposition]{Corollary}
\newtheorem{theorem}[proposition]{Theorem}

\theoremstyle{remark}

\newcommand{\PP}{\mathbb{P}}
\newcommand{\C}{\mathbb{C}}

\newcommand{\Q}{\mathbb{Q}}
\newcommand{\Z}{\mathbb{Z}}
\newcommand{\QQ}{\mathcal{Q}}
\newcommand{\DD}{\mathcal{D}}
\newcommand{\XX}{\mathcal{X}}
\newcommand{\OO}{\mathcal{O}}
\newcommand{\UU}{\mathcal{U}}
\newcommand{\HH}{\mathcal{H}}
\newcommand{\HHH}{\mathbb{H}}
\newcommand{\VV}{\mathcal{V}}

\newcommand{\NS}{\mathrm{NS}}
\newcommand{\NL}{\mathrm{NL}}
\newcommand{\KK}{\mathcal{K}}

\newcommand{\aaa}{\mathbf{a}}
\newcommand{\bbb}{\mathbf{b}}
\newcommand{\ccc}{\mathbf{c}}

\newcommand{\sing}{\mathrm{sing}}
\newcommand{\sm}{\mathrm{smooth}}

\usepackage{titlesec}

\usepackage{setspace}
\titleformat*{\section}{\large\scshape}
\titleformat*{\subsection}{\scshape}
\titleformat*{\subsubsection}{\scshape}

\title{\normalsize{\textbf{ONE-DIMENSIONAL LOCAL FAMILIES OF COMPLEX K3 SURFACES}}}
\author{\normalsize{\textsc{Riccardo Carini \ \& \ Francesco Viganò}}}
\date{}

\begin{document}
\sloppy

\maketitle

\begin{abstract}
	\noindent\small{For any complex K3 surface $X$, we construct a one-dimensional deformation in which all integers $\rho$ with $0 \leqslant \rho \leqslant 20$ occur as Picard numbers of some fibres. In contrast, we prove that the \textit{generic} one-dimensional local family of K3 surfaces admits only $0$ and $1$ as Picard numbers of the fibres.}
\end{abstract}

We consider families of complex K3 surfaces $\mathcal{X} \to S$ over a connected base $S$, and the variation of the Picard number $\rho(\XX_t)$ of the fibre $\XX_t$, i.e. the rank of the Néron-Severi group $\NS(\XX_t)$. For a complex K3 surface $X$, $0\leqslant \rho(X)\leqslant 20$, and $\rho(X)\geqslant 1$ if $X$ is algebraic or, equivalently, projective. Given a family $\XX \to S$, define $\rho_{\min}$ as
\[
\rho_{\min}=\min \Set{ \rho(X_t) | t\in S }.
\]
Then, the Noether-Lefschetz locus
\[
\NL(\XX/S)=\Set{ t\in S | \rho(X_t)>\rho_{\min} }
\]
is a dense subset of the base if the family is not isotrivial \cite[Theorem 1.1]{Og}. Moreover, it is a countable union of subvarieties of $S$ of positive codimension -- in particular, $\NL(\XX/S)$ is countable if $S$ is one-dimensional. As a consequence, any family with constant Picard number is necessarily isotrivial (see \cite[Theorem 1.1]{SB} for a weaker original version). It is interesting to understand the behaviour of the Picard number of the fibres of non-isotrivial families: not only where it jumps, but also how to quantify its jump $\rho-\rho_{\min}$. Similar questions were also studied in the context of arithmetic specializations, the base being $\mathrm{Spec} (\OO_E)$ for some number field $E$ (see for instance \cite{Cha}).\\

The universal family of marked K3 surfaces -- or, in the local case, the Kuranishi family of any given K3 surface \cite[Application 1.3]{Og} -- provides an example where each possible Picard number $\rho$ with $0\leqslant\rho\leqslant20$ is attained. However, the base $S$ that parametrises such families is $20$-dimensional. Several known examples (see \textsection \ref{section.examples}) show that the set of Picard values can be very limited. Nevertheless, these families are somehow special, and one may wonder how the Picard number is altered in a general case. We prove that any K3 surface admits a one-dimensional deformation with the same property (Theorem \ref{theorem.result}):

\begin{theorem}\label{theorem.result0.1}
	Given a K3 surface $X$, there is a one-dimensional deformation $\mathcal{X}\to\Omega$ of $X$ with the property that any integer $\rho$ with $0\leqslant\rho\leqslant20$ is realised as the Picard number of a fibre.
\end{theorem}

\begin{center}
	\includegraphics[scale=0.4]{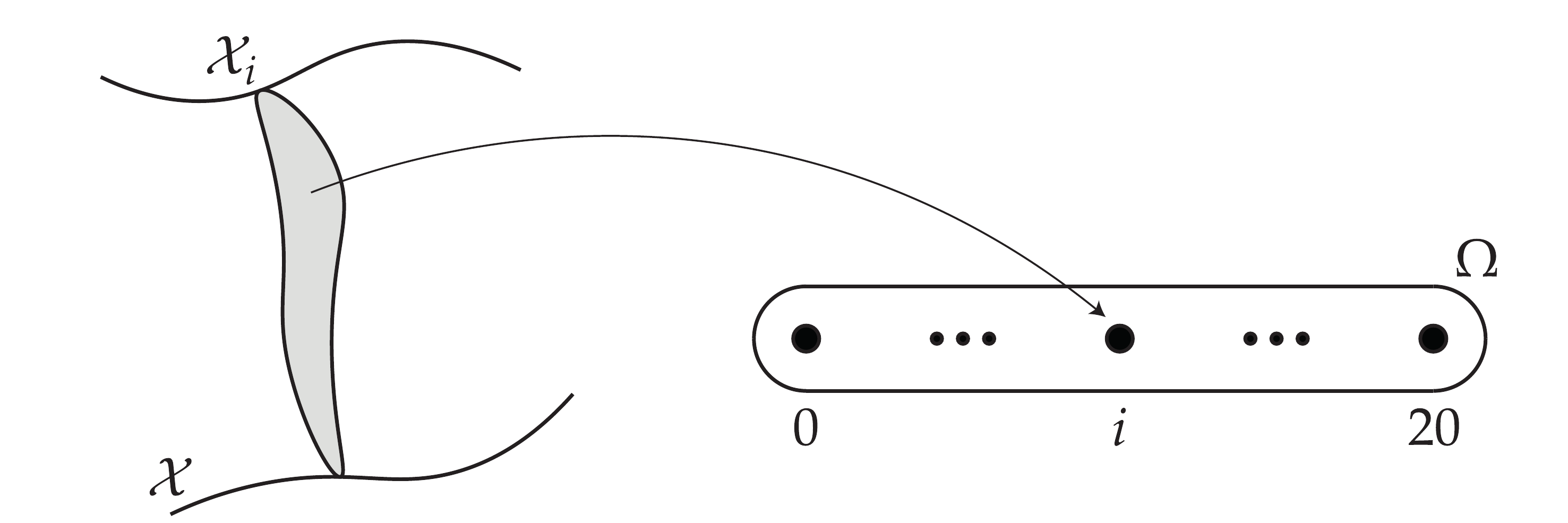}
\end{center}

The one-dimensional base $\Omega$ is, in our construction, an open neighbourhood in $\C$ of the real interval $[0,20]$, and $\rho(\XX_i)=i$ for each $i=0,\dots,20$. The original $X$ is the fibre $\XX_i$ for $i=\rho(X)$. If the original K3 surface $X$ is algebraic, one can also build a deformation whose fibres are all algebraic, for which all integers between $1$ and $20$ are attained as Picard numbers of some fibres (Theorem \ref{theorem.resultalgebraic}).\\

The strategy for constructing this family is essentially to find a suitable curve through points in a moduli space corresponding to K3 surfaces with the desired Picard numbers. The interpolation process is first carried out in the K3 period domain -- that is, at a Hodge-theoretical level  --, and then lifted to a moduli space via Torelli.\\

However, the family we construct is somehow special, as we prove that for most one-dimensional local families, the range of possible Picard numbers is very limited (Theorem \ref{theorem.01rho}):

\begin{theorem}\label{theorem.result0.2}
	The generic one-dimensional local family of K3 surfaces admits only Picard numbers $0$ and $1$.
\end{theorem}

By \textit{one-dimensional local family} we mean (as in \cite{Og}) a holomorphic family parametrised by an open one-dimensional disk. Hence, we regard the space of such one-dimensional local families as an open subset of a complete metric space, namely the space of holomorphic maps from the unit disk to a suitable open subset of $\mathbb{C}^{20}$. We say that a property holds \textit{generically} if it holds in the complement of a nowhere dense subset (see \textsection \ref{section.vg} for a more precise treatment). Similarly, if we only consider families with algebraic fibres, the generic one-dimensional local family has Picard numbers $1$ and $2$. In a similar fashion, our results can be adapted to hyperk\"ahler manifolds.\\

\textbf{Acknowledgements.} The authors are grateful to Richard Thomas for many valuable discussions and his continuous support. They also thank Daniel Huybrechts for useful conversations that motivated this project.

This work was supported by the Engineering and Physical Sciences Research Council [EP/S021590/1]. The EPSRC Centre for Doctoral Training in Geometry and Number Theory (The London School of Geometry and Number Theory), University College London. The authors are PhD candidates at Imperial College London.

\section{Examples of constrained variation of $\rho(\XX_t)$}\label{section.examples}

Before moving towards our result, we discuss one of the reasons that motivates our interest in one-dimensional families. In fact, in many examples the Picard number does not swipe out a full range of integers, and is constrained to a few possible values. We illustrate this phenomenon in various scenarios.\\

For a complex number $\tau \in \HHH$, denote by $E_\tau$ the correspondent elliptic curve, and by $\KK(\tau, \tau')$ the Kummer surface associated with the abelian surface $E_\tau \times E_{\tau'}$. Then (see \cite[Chapter 17, \textsection 1.4]{Huy}),
\[
\rho(\KK(\tau,\tau'))=\begin{cases}
	18 & \text{ if }E_{\tau}\nsim E_{\tau'},\\
	19 & \text{ if }E_{\tau}\sim E_{\tau'}\text{ without CM},\\
	20 & \text{ if }E_{\tau}\sim E_{\tau'}\text{ with CM},
\end{cases}
\]
where we write $E_{\tau}\sim E_{\tau'}$ if the curves are isogenous. Indeed, $\rho(\KK(\tau,\tau'))\geqslant16$ as $\text{NS}(K(\tau,\tau'))$ contains the classes of the 16 exceptional divisors coming from the Kummer construction. Moreover, the fibres of the two projections onto $E_\tau$ and $E_{\tau'}$ provide two extra classes, so that $\rho(K(\tau,\tau'))\geqslant 18$. Finally, if $E_{\tau}\sim E_{\tau'}$, the graph of an integer multiplication and -- possibly, if $E_\tau$ is CM -- a complex multiplication give the remaining classes. Explicitly, we can consider three families:

\begin{itemize}
	\item Fix $\zeta\in\HHH$ with $[\mathbb{Q}(\zeta):\mathbb{Q}]>2$, so that $E_{\zeta}$ is not CM. The one-dimensional family $\KK(\tau,\zeta)$ has $\rho_{\min}=18$, and $\rho(\KK(\tau,\zeta))=19$ for a dense and countable subspace of values of $\tau$, while $20$ is never attained. To be precise, the Picard number jumps if and only if $\tau \in \Q \oplus \Q \zeta$, that is dense and countable in the base.
	\item Let $\tau=\tau'$, and consider the one-dimensional family $\KK(\tau,\tau)$. The fibre has Picard number $19$ or $20$, when $E_\tau$ is not or is CM, respectively. Again, the Picard number jumps over a dense and countable subset of the base.
	\item The family $\KK(\tau,i)$ has $\rho_{\min}=18$ and $\rho(X_{\tau})=20$ for $\tau\in\mathbb{Q}(i)$, while $19$ is not realised. Notice that, again, the Noether-Lefschetz locus is countable and dense in the base.
\end{itemize}

In fact, the latter instance falls under a more general picture. In a family $\XX \to S$ of algebraic K3 surfaces, the endomorphism fields $K_{T(\XX_t)}$ of the transcendental lattices $T(\XX_t)$ can be compared with each other via the Zarhin embedding $K_{T(\XX_t)} \hookrightarrow \C$ (see \cite{Zar}, or \cite[Chapter 3, \textsection 3.3]{Huy}, for the description of the embedding). If $K$ is a subfield of each of the $K_{T(\XX_t)}$, then the Picard number can only jump by multiples of $[K:\Q]$ (see \cite[Chapter 17, Remark 1.4]{Huy}). In the specific case of $\KK(\tau,i)$, the CM map given by the multiplication by $i$ on $E_i$, that acts as the multiplication by $i$ on $H^{1,0}(E_{i})$ as well, lifts to the associated Kummer surface, and acts as multiplication by $i$ on $H^{2,0}(\KK(\tau,i))$, for any $\tau$. Therefore, the Picard jump is a multiple of $[\Q(i):\Q]=2$.\\

We now present another example: the twistor space of a K3 surface (see \cite[Section 3.F]{HKLR}, \cite{Hit} or \cite[Chapter 7]{Joyce} for more on the twistor space and the details of its construction). Given a complex projective K3 surface $X$ (and a choice of an ample class), its twistor space is a 3-fold $\mathcal{X}$, together with a natural holomorphic projection $\mathcal{X}\to\mathbb{P}^{1}\simeq S^{2}$. The fibres are again K3 surfaces, and only countably many of them are algebraic. The set of possible Picard numbers is very constrained: most fibres have Picard number $\rho_{\min}=\rho(X)-1$, and outside the equator the Picard number jumps only to $\rho(X)$, on a countable and dense subset \cite[Proposition 3.2]{Huy}. On the equator, it can only jump by multiples of $[K_{T}^{0}:\Q]$, where $K_{T}$ is, as above, the endomorphism field of $T(X)$, and $K_{T}^{0}=K_{T}\cap\mathbb{R}$. In particular, if $X$ is CM, the only possible Picard values of the twistor fibres are $\rho(X)-1,\rho(X)$ and $10+\frac{\rho(X)}{2}$ \cite[Theorem 5.3 and Remark 3.7]{Vig}. More generally, Huybrechts \cite{HuyBr} describes the behaviour of the Picard number in a specific class of deformations, named brilliant families of K3 surfaces. Among these, (the upper half sphere of) the twistor line, the Brauer family, and the Dwork pencil. The possible values are only $\rho(X)$ and $\rho(X)-1$, where $\rho(X)$ is the Picard number of the central fibre.

\section{The K3 lattice $\Lambda$}

All K3 surfaces are diffeomorphic to each other, and therefore have the same cohomology groups. A marking on a K3 surface $X$ is a choice of isomorphism of lattices $H^2(X,\Z) \simeq \Lambda$, which then respects the intersection form $( \ . \ )$. If $\sigma$ is a $(2,0)$-form on a K3 surface, the corresponding period point is defined as $[\sigma] \in \PP(\Lambda \otimes \C)$. However, as $\sigma$ is constrained by some conditions, the period point lies in the period domain
\[
\DD= \Set{ [\sigma] \in \PP(\Lambda \otimes \C) | (\sigma.\sigma)=0, \ (\sigma.\overline{\sigma})>0 },
\]
that is an open subset of the quadric $\QQ$ defined by the condition $(\sigma.\sigma)=0$.\\

Assume that $\{\gamma_j\}_{j=1}^{22}$ is a $\Z$-basis of $H^2(X,\Z)$ that gives the marking. Let $\{\gamma_j^\vee\}_{j=1}^{22} \subseteq H^2(X,\Q)$ be the correspondent $\Q$-dual basis of $H^2(X,\Q)$ with respect to the intersection form $( \ . \ )$. Note that $\sigma=\sum_{j=1}^{22} (\sigma.\gamma_j) \gamma_j^\vee$. Then, the dual basis gives an isomorphism $\PP(\Lambda \otimes \C) \simeq \PP^{21}_\C$ for which the point $[\sigma]$ corresponds to the coordinate vector of its periods
\[
[ (\sigma.\gamma_1) : \dots : (\sigma.\gamma_{22}) ] \in \PP^{21}_\C.
\]
We call this vector $v$, so that $v^j=(\sigma.\gamma_j)$. The index notation we use for projective coordinates in $\PP^{21}_\C$ is $[v^1 : \dots : v^{22}]$, rather than $[v^0 : \dots : v^{21}]$. \\

For such a vector $v$, we define
\[
\chi(v)=\dim_{\Q} \langle v^j \rangle^{j=1,\dots,22}_{\Q}
\]
The relation between the Picard number of a K3 surface and the $\chi$-value of its period point is here made explicit.

\begin{lemma}\label{lemma.rhochi}
	Let $v \in \DD$ be the period point of the K3 surface $X$. Then
	\[
	\rho(X)=22-\chi(v).
	\]
\end{lemma}
\begin{proof}
	The Picard number of $X$ is the rank of its Néron-Severi group, or equivalently
	\[
	\rho(X)= \dim_\Q \NS(X) \otimes \Q = \dim_\Q \Set{ \gamma \in \Lambda \otimes \Q | (\sigma.\gamma)=0 }
	\]
	from Lefschetz $(1,1)$-theorem (see \cite[Chapter 1, Section 2]{GH}). Thus, the equality $\rho(X)=22-\chi(v)$ follows immediately from rank-nullity theorem.
\end{proof}

\section{A special deformation}

Following Lemma \ref{lemma.rhochi}, we are led to the search of $21$ points in $\DD$, one of which coinciding with the period point of the original K3, whose $\chi$-values cover all integer numbers between $2$ and $22$.

\begin{lemma}\label{lemma.points}
	For any open $\UU$ in $\DD$, there exist $21$ points $\{v_i\}_{i=0}^{20}$ in $\UU$ satisfying $\chi(v_i)=22-i$.
\end{lemma}
\begin{proof}
	The result is stated, in an alternative form, in \cite[Application 1.3]{Og}. However, we outline here a direct proof. We find these points recursively, starting from $v_0$. First, notice that, $\chi(v)\leqslant 21$ if and only if there exists a non-zero vector $\aaa \in \Q^{22}$ such that $\aaa \cdot v =0$. Thus, the locus of $v$ such that $\chi(v) \leqslant 21$ is a countable union of hyperplanes in $\PP_\C^{21}$. As $\UU$ is an open subset of the smooth quadric $\QQ$, $\UU$ is not contained in any of them, and therefore there exists a point $v_0 \in \UU$ satisfying $\chi(v_0)=22-0=22$. However, as the union of these hyperplanes is dense in $\PP_\C^{21}$, there exists an $\aaa \in \Q^{22}$ whose corresponding hyperplane $H_\aaa$ cuts $\UU$, and thus any point $v \in \UU \cap H_\aaa$ will have $\chi(v)\leq 21$. Notice that, up to a small adjustment, we may assume that $(\aaa.\aaa)\neq 0$. We focus our attention on the hyperplane $H_\aaa$. The restriction of the form $( \ . \ )$ to $H_\aaa$ is non-degenerate, thanks to the choice of $\aaa$, and therefore the quadric $\QQ \cap H_\aaa$ is smooth. Hence, the argument can be recursively applied to find the aimed points.
\end{proof}

Lemma \ref{lemma.points} gives a recipe to find points with prescribed $\chi$-values; the next result \textit{connects the dots} by means of a holomorphic curve. In order to do so, we will make use of the Lagrange polynomials to interpolate these points. For $i=0,\dots, 20$, consider the Lagrange polynomial
\[
p_i(t) = \frac{ \prod_{j \neq i} (t-j) }{ \prod_{j\neq i} (i-j)},
\]
where the products are taken over $j=0,\dots, 20$. For $k=0,\dots,20$, we have $p_i(k)=\delta_{ik}$. Also, we define $M$ by
\[
M=\max_{i=0,\dots,20} \left( \max_{t \in [0,20]} | p_i(t) | \right).
\]

\begin{lemma}\label{lemma.curve}
	Let $v \in \DD$. For each neighbourhood $\VV$ of $v$ in $\DD$, there exist an open neighbourhood $\Omega$ of the real interval $[0,20]$ in $\C$, and a holomorphic curve $\tau \colon \Omega \to \VV$ such that
	\[
	\chi(\tau(i))=22-i
	\]
	for $i=0,\dots, 20$. The curve can be chosen so that $\tau(i)=v$ for $i=22-\chi(v)$.
\end{lemma}
\begin{proof}
	First, consider -- up to shrinking $\VV$ -- a chart $q \colon \VV \to B(0,1) \subseteq \C^{20}$ for which $q(v)=0$. Set $r=1/(21\cdot M)$, and $\UU=q^{-1}(B(0,r))\subseteq \VV$.
	
	\begin{center}
	\includegraphics[scale=0.4]{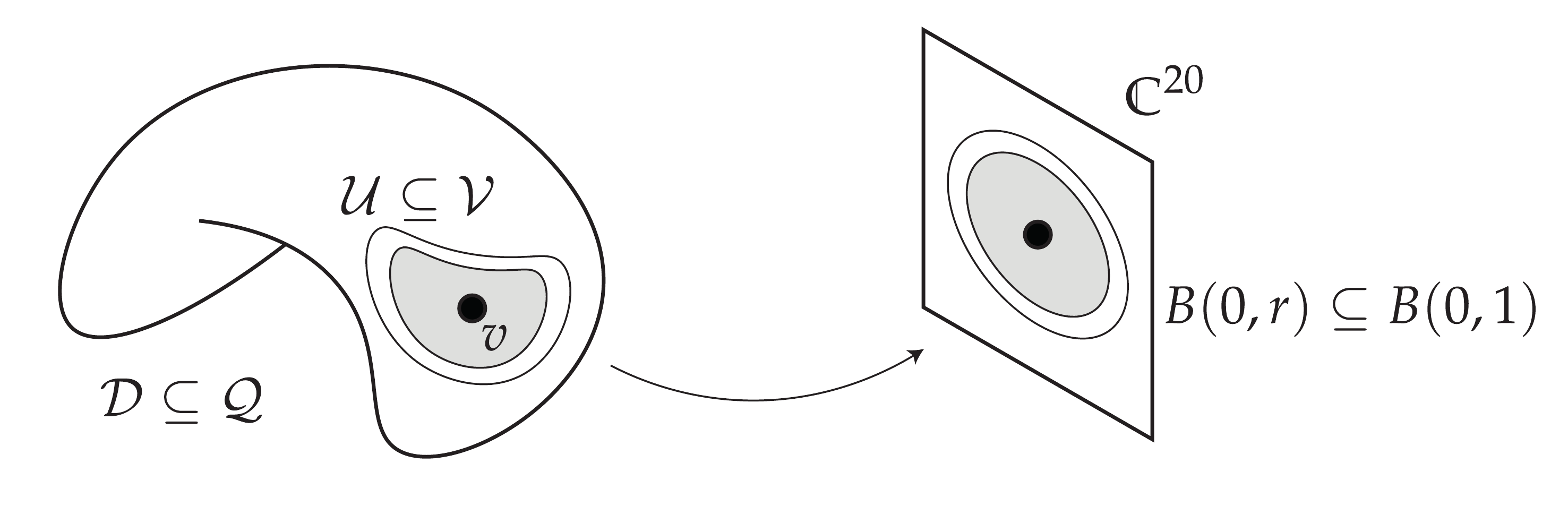}
	\end{center}
	
	Thanks to Lemma \ref{lemma.points}, there exist $21$ points $\{v_i\}_{i=0}^{20}$ in $\UU$ satisfying $\chi(v_i)=22-i$. Of course, we can suppose $v_i=v$ for $i=22-\chi(v)$. Now, consider the function $\tilde{\tau}$ in the complex variable $t$ defined as
	\[
	\tilde{\tau}(t)= \sum_{i=0}^{20} p_i(t) q(v_i).
	\]
	Note that $\tilde{\tau}(i)=q(v_i)$ for $i=0,\dots,20$. Then, if $t$ belongs to the real interval $[0,20]$, we have
	\[
	| \tilde{\tau}(t) | \leqslant \sum_{i=0}^{20} |p_i(t) q(v_i)| < 21 \cdot M \cdot r = 1.
	\]
	\begin{center}
	\includegraphics[scale=0.4]{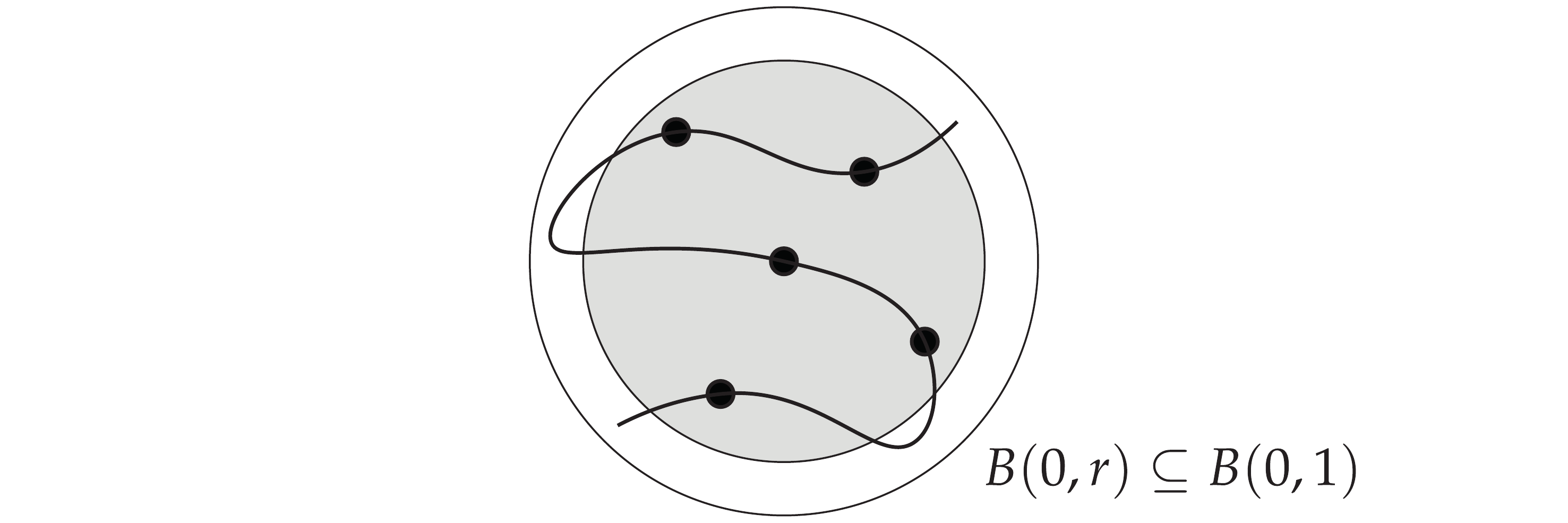}
	\end{center}
	Therefore, there exists an open neighbourhood $\Omega$ of $[0,20]$ in $\C$ such that $\tilde{\tau}(\Omega) \subseteq B(0,1)$. The holomorphic curve $\tau=q^{-1}\circ \tilde{\tau} \colon \Omega \to \VV$ satisfies the requirements.
\end{proof}

We state and prove the announced Theorem \ref{theorem.result0.1}:

\begin{theorem}\label{theorem.result}
	Given a K3 surface $X$, there is a one-dimensional deformation $\mathcal{X}\to\Omega$ of $X$ with the property that any integer $\rho$ with $0\leqslant\rho\leqslant20$ is realised as the Picard number of a fibre.
\end{theorem}

\begin{center}
	\includegraphics[scale=0.4]{fibration.pdf}
\end{center}

\begin{proof}
	A complex K3 surface $X$ admits a smooth universal deformation $\mathcal{X}\to\text{Def}(X)$ -- the Kuranishi family -- whose base $\text{Def}(X)$ can be taken to be a disk in $\mathbb{C}^{20}$. In particular, the period map $\text{Def}(X)\to\mathcal{D}$ is well defined and, by local Torelli (\cite[Theorem VIII.7.3]{Bar}), a local isomorphism. Choosing a suitable neighbourhood $\VV\subseteq\mathcal{D}$ of the period point of $X$, for which the restriction of the period map is an isomorphism onto its image, we can lift the curve constructed in Lemma \ref{lemma.curve} to $\text{Def}(X)$, thus getting the desired one-dimensional family.
\end{proof}

As already mentioned, a similar result holds for algebraic K3 surfaces.

\begin{theorem}\label{theorem.resultalgebraic}
	Given a algebraic K3 surface $X$, there is a one-dimensional deformation $\mathcal{X}\to\Omega$ of $X$ in algebraic K3 surfaces with the property that any integer $\rho$ with $1\leqslant\rho\leqslant20$ is realised as the Picard number of a fibre.
\end{theorem}

\begin{proof}
	Let $\ell$ be an ample class on $X$. It is easy to rephrase Lemma \ref{lemma.points} and Lemma \ref{lemma.curve} so that $20$ points $\{v_i\}_{i=1}^{20}$ are chosen in an open $\UU$ of $\DD \cap \ell^\perp$, and an analogous curve can be constructed. The proof is then the same as for Theorem \ref{theorem.result}, and the class $\ell$ is ample for each fibre -- being $\ell$ a rational $(1,1)$-class with $(\ell.\ell)>0$ \cite[Theorem IV.6.2]{Bar} --, so that each fibre is algebraic.
\end{proof}

\section{The generic local family}\label{section.vg}

In this section we prove that the \textit{generic} one-dimensional local family of K3 surfaces admits only $0$ and $1$ as Picard numbers of the fibres. 
This result (Theorem \ref{theorem.01rho}) underlines the peculiarity of the construction of Theorem \ref{theorem.result}. \\

By local Torelli (\cite[Theorem VIII.7.3]{Bar}) -- up to shrinking the base -- we can think of a one-dimensional local family of K3 surfaces (as in the statement of Theorem \ref{theorem.result0.2}) as a holomorphic map from the unit disk $\Delta\subset \mathbb{C}$ to the period domain $\mathcal{D}$. By further shrinking $\Delta$ if necessary, we can consider holomorphic maps $\tau: \Delta \to B$, where $B$ is an open ball in $\mathbb{C}^{20}$, and denotes -- by abuse of notation -- both an open subset of $\mathcal{D}$ and its image in $\mathbb{C}^{20}$ via a holomorphic chart of $\mathcal{D}$. We may also assume that such maps are continuously defined on the closed disk $\overline{\Delta}$.  Hence, we are interested in the following space
$$ \mathcal{H} = \operatorname{Hol}(\Delta, B)\cap \mathcal{C}(\overline{\Delta}, B)  $$
of holomorphic maps $\tau: \Delta\to B$ which are holomorphic and continuous on the closed disk $\overline{\Delta}$. \\

Since $\mathcal{H}$ is an open subset of the complete metric space $\operatorname{Hol}(\Delta, \mathbb{C}^{20})\cap \mathcal{C}(\overline{\Delta}, \mathbb{C}^{20})$, endowed with the sup-norm, we have a natural notion of topologically negligible subset of $\mathcal{H}$. 
Namely, we say a subset $S\subset \mathcal{H}$ is \textit{generic} if its complement is meagre, i.e. a countable union of nowhere dense subsets  -- often in the literature co-meagre or residual are used instead. Equivalently, $S$ contains a countable intersection of open dense subsets. Since $\mathcal{H}$ is a Baire space, any such $S$ is dense in $\mathcal{H}$. \\

In this setup, the strategy to prove Theorem \ref{theorem.result0.2} is the following: by Lemma \ref{lemma.rhochi} a curve $\tau\in \mathcal{H}$ corresponds to a one-dimensional local family with only Picard number 0 and 1 if it does not intersect the subset of $B\subset \mathcal{D}$ where $\chi \leq 20$. Hence we first prove that the locus we want to avoid
$$ \{v\in B \mid \chi(v)\leq 20 \} $$
is a countable union of codimension 2 analytic subvarieties of $B$ (Corollary \ref{bad_locus}). Then we prove that the subset of curves in $\mathcal{H}$ hitting only one of such subvarieties is nowhere dense in $\mathcal{H}$ (Lemma \ref{lemma.codim2}).

\begin{lemma}\label{lemma.codim2}
	Let $H_{\mathbf{a}}$ and $H_{\mathbf{b}}$ be two different hyperplanes and $\mathcal{Q}$ a non-singular quadric in $\mathbb{P}_{\mathbb{C}}^{n}$. If $n\geqslant4$, then $H_{\mathbf{a}}\cap H_{\mathbf{b}}\cap\mathcal{Q}$ has codimension two in $\mathcal{Q}$ and its singular locus is a smooth submanifold of $H_{\mathbf{a}}\cap H_{\mathbf{b}}\cap \mathcal{Q}$. 
\end{lemma}
\begin{proof}
	As $\QQ$ is non-singular, the intersection $H_\aaa \cap \QQ$ is of codimension one in $\QQ$. Therefore, the only case to exclude is that the codimension does not increase when considering the intersection with $H_\bbb$. This happens if and only if one of the irreducible components of $\QQ \cap H_\aaa$ is completely contained in $H_\bbb$. In particular, the degree-$2$ polynomial describing $\QQ \cap H_\aaa$ should split as a product of two linear polynomials, one of which corresponding to $\bbb$. Thus, this polynomial would be of the form $(\sum_i b_i x_i)\cdot (\sum_j c_j x_j)$ for some $\ccc \in \C^{n+1}$, with symmetric matrix
	\[
	\frac{\bbb \cdot \ccc^T + \ccc \cdot \bbb^T}{2}.
	\]
	This matrix has rank at most $2$, being the sum of two matrices of rank at most $1$. However, it is easy to show that, if $\QQ$ is a non-singular (i.e., of rank $n+1$) quadric in $\PP^n_\C$, then the rank of $\QQ \cap H_\aaa$ is either $n$ or $n-1$. This is a contradiction, thanks to the hypothesis $n\geqslant 4$. 
	
	The last assertion follows from the fact that $H_{\mathbf{a}}\cap H_{\mathbf{b}}\cap \mathcal{Q}$ is a -- possibly singular -- quadric in $H_{\mathbf{a}}\cap H_{\mathbf{b}}$, and hence its singular locus is a linear subspace of $H_{\mathbf{a}}\cap H_{\mathbf{b}}$ contained in $H_{\mathbf{a}}\cap H_{\mathbf{b}}\cap \mathcal{Q}$, and is hence smooth.
\end{proof}

\begin{corollary}\label{bad_locus}
The locus $\{v\in B \mid \chi(v)\leq 20\}$ is a countable union of codimension 2 analytic subvarieties of $B$, each of which has smooth singular locus.
\end{corollary}

\begin{proof}
    Note that for $v\in \mathcal{D}$, one has $\chi(v)\leq 20$ if and only if there exists two different hyperplanes $H_\mathbf{a}$ and $H_{\mathbf{b}}$, for some $\mathbf{a},\mathbf{b}\in \mathbb{Q}^{22}\setminus \{0\}$, such that $v\in H_\mathbf{a}\cap H_\mathbf{b}$. Hence
    \begin{equation*} 
    \{v\in B\mid \chi(v)\leq 20\}= \bigcup_{\substack{\mathbf{a}, \mathbf{b}\in \mathbb{Q}^{22} \\ \dim_\mathbb{Q} \langle \mathbf{a},\mathbf{b} \rangle=2 }} (B\cap H_\mathbf{a}\cap H_\mathbf{b}), 
    \end{equation*}
    so that the claim follows from Lemma \ref{lemma.codim2}.
\end{proof}

\begin{lemma}\label{codim2_nowhere_dense}
    Let $M\subset B$ be an analytic subvariety of codimension $2$ and assume that the singular locus $M_{\textnormal{sing}}\subset M$ is a smooth submanifold of $B$. Then the subset
    $$ \mathcal{G}_M = \left\lbrace \tau \in \mathcal{H} \mid \tau ( \overline{\Delta} ) \cap M \neq \emptyset \right\rbrace $$
    given by curves intersecting $M$ is nowhere dense in $\mathcal{H}$.
\end{lemma}

\begin{proof} Note that $\mathcal{G}_M$ is closed as $\overline{\Delta}$ is compact and $M$ is closed, so that it is enough to prove that its complement $\mathcal{H} \setminus \mathcal{G}_M$ is dense. 
Denote by $M_\sm = M\setminus M_\sing$. 
We will prove that the following inclusions are open dense immersions:
\begin{equation*}
    \begin{split}
        \HH\setminus \mathcal{G}_M & = \Set{\tau \in \HH | \tau(\overline{\Delta})\cap M_\sm = \emptyset \ \text{and} \ \tau(\overline{\Delta})\cap M_\sing = \emptyset } \\
        & \subseteq \Set{\tau \in \HH | \tau(\partial \Delta)\cap M_\sm = \emptyset \ \text{and} \ \tau(\overline{\Delta})\cap M_\sing = \emptyset } \\
        & \subseteq \Set{\tau\in \HH | \tau(\overline{\Delta})\cap M_\sing =\emptyset } \\
        & \subseteq \Set{\tau\in \HH | \tau(\partial \Delta)\cap M_\sing =\emptyset } \\
        & \subseteq \Set{\tau\in \HH | \tau(\overline{\Delta}) \nsubseteq M_\sing  } \subseteq \HH.
    \end{split}
\end{equation*}

More explicitly, starting with $\tau \in \HH$, we approximate it step by step with a curve lying in the smaller subset in the above chain of inclusions:
\begin{itemize}
    \item If $\tau(\overline{\Delta})$ is contained in $M_\sing$, then for all $\varepsilon$ small enough the image of the curve $\tau(t) + \varepsilon v$ is not all contained in $M_\sing$, where $v$ is a vector not belonging to the tangent space to $M_\sing$ at a point of $\tau(\Delta)$.
    \item In order to get $\tau(\partial \Delta)\cap M_\sing =\emptyset$, consider $\tau(R_n t)$ in lieu of $\tau$, where $R_n$ is a sequence of radii converging to $1$ from below, for which $\tau$ does not hit $M_\sing$ on the boundary of the disks of radii $R_n$. Notice that the sequence $R_n$ always exists as $\tau$ hits $M_\sing$ only at countably many times in $\Delta$.
    \item Now to get $\tau(\overline{\Delta})\cap M_\sing =\emptyset$, we note that this is equivalent to asking that $\tau$ is transversal to $M_\sing$ (according to \cite[\textsection 1.5]{difftop}), as $M_\sing$ has codimension greater than 1. For this, let $B'$ a sufficiently small ball in $\C^{20}$, so that $\tau(t)+v$ is contained in $B$ for all $t \in \bar{\Delta}$, $v \in B'$. From \cite[Transversality Theorem, \textsection 2.3]{difftop} it follows that the set of $v \in \C^{20}$ of small norm for which $\tau(t) + v$ is not transversal to $M_\sing$ is of measure zero in $B'$. As this set is also closed in $B'$ (by openness of transversality), it is nowhere dense in $B'$. Therefore, we can choose a small $v$ such that the associated curve $\tau(t)+v$ is transversal to $M_\sing$, and hence does not intersect $M_\sing$.
    \item To avoid $\tau$ intersecting $M_\sm$ on the boundary, we argue as done for $M_\sing$. Notice that, as $\tau$ is far from $M_\sing$ already, a small perturbation is far from $M_\sing$ as well.
    \item Finally, in order to get to $\HH\setminus \mathcal{G}_M$, we may assume that every small perturbation of $\tau$ is far from $M_\sing$. And then we argue as done for $M_\sing$ using transversality. \qedhere 
\end{itemize} 
\end{proof}

We are finally ready to prove Theorem \ref{theorem.result0.2}:

\begin{theorem}\label{theorem.01rho}
    The generic one-dimensional local family of K3 surfaces admits only Picard numbers 0 and 1.
\end{theorem}

\begin{proof}
    As in the beginning of this section, we think of $\mathcal{H}$ as the space parametrising one-dimensional local families of K3 surfaces. Then we only have to prove that the subset
    $$ \{\tau \in \mathcal{H} \mid \chi(\tau(t))\leq 20\} $$
    is a countable union of nowhere dense subsets of $\mathcal{H}$. According to Corollary \ref{bad_locus} and using the notation of Lemma \ref{codim2_nowhere_dense}, we have
    $$   \{\tau \in \mathcal{H} \mid \chi(\tau(t))\leq 20\}= \bigcup_{\substack{\mathbf{a}, \mathbf{b}\in \mathbb{Q}^{22} \\ \dim_\mathbb{Q} \langle \mathbf{a},\mathbf{b} \rangle=2 }} \mathcal{G}_{B\cap H_\mathbf{a}\cap H_\mathbf{b}},  $$
    and hence we can conclude by Corollary \ref{bad_locus} and Lemma \ref{codim2_nowhere_dense}.
\end{proof}

Using the same strategy as in the proof of Theorem \ref{theorem.resultalgebraic}, we can adapt the previous result to local families of algebraic K3 surfaces:

\begin{theorem}
    The generic one-dimensional local family of algebraic K3 surfaces admits only Picard numbers 1 and 2.
\end{theorem}

\newpage
\printbibliography

\end{document}